\nonstopmode
\documentclass[reqno,10pt]{amsart}
\usepackage{latexsym}
\usepackage{fancyhdr}
\usepackage{amsmath, amssymb}
\usepackage[ansinew]{inputenc}
\usepackage[all]{xy}
\usepackage{pdflscape}
\usepackage{longtable}
\usepackage{rotating}
\usepackage{verbatim}
\usepackage{hyperref}
\usepackage{subfigure}
\usepackage{mathrsfs}
\usepackage{cleveref}
\usepackage{mdwlist}
\usepackage{color}

\newcounter{mnotecount}[section]

\newcommand{\rmnote}[1]{}

\DeclareFontFamily{U}{mathb}{\hyphenchar\font45}
\DeclareFontShape{U}{mathb}{m}{n}{
      <5> <6> <7> <8> <9> <10> gen * mathb
      <10.95> mathb10 <12> <14.4> <17.28> <20.74> <24.88> mathb12
      }{}
\DeclareSymbolFont{mathb}{U}{mathb}{m}{n}
\DeclareFontSubstitution{U}{mathb}{m}{n}

\let\dot\relax
\DeclareMathAccent{\dot}{0}{mathb}{"39}
\let\ddot\relax
\DeclareMathAccent{\ddot}{0}{mathb}{"3A}
\let\dddot\relax
\DeclareMathAccent{\dddot}{0}{mathb}{"3B}
\let\ddddot\relax
\DeclareMathAccent{\ddddot}{0}{mathb}{"3C}

\theoremstyle{plain}
\newtheorem*{theorem*}{Theorem}
\newtheorem{theorem}{Theorem}
\newtheorem*{lemma*}{Lemma}
\newtheorem{lemma}[theorem]{Lemma}
\newtheorem*{proposition*}{Proposition}

\newtheorem*{corollary*}{Corollary}

\newtheorem*{claim*}{Claim}

\newtheorem*{conjecture*}{Conjecture}

\newtheorem*{question*}{Question}
\theoremstyle{definition}
\newtheorem*{definition*}{Definition}

\newtheorem*{example*}{Example}

\newtheorem*{algorithm*}{Algorithm}
\newtheorem*{remark*}{Remark}
\newtheorem*{remarks*}{Remarks}
\newtheorem{remark}[theorem]{Remark}

\newtheorem*{convention*}{Convention}

\newcommand{\al}{\alpha}
\newcommand{\be}{\beta}

\newcommand{\de}{\delta}
\newcommand{\ep}{\epsilon}

\newcommand{\rh}{\rho}

\newcommand{\si}{\sigma}

\newcommand{\ta}{\tau}

\newcommand{\vh}{\varphi}

\newcommand{\om}{\omega}

\newcommand{\Ph}{\Phi}

\newcommand{\N}{\mathbb{N}}

\newcommand{\R}{\mathbb{R}}

\newcommand{\cA}{\mathcal{A}}

\newcommand{\cC}{\mathcal{C}}

\newcommand{\cE}{\mathcal{E}}

\newcommand{\fM}{\mathfrak{M}}

\newcommand{\p}{\partial}

\renewcommand{\o}{\circ}

\newcommand{\on}{\operatorname}

\newcommand{\sr}[1]%
{\ifmmode{}^\dagger\else${}^\dagger$\fi\ifvmode
\vbox to 0pt{\vss
 \hbox to 0pt{\hskip\hsize\hskip1em
 \vbox{\hsize3cm\raggedright\pretolerance10000
 \noindent #1\hfill}\hss}\vss}\else
 \vadjust{\vbox to0pt{\vss%
 \hbox to 0pt{\hskip\hsize\hskip1em%
 \vbox{\hsize3cm\raggedright\pretolerance10000%
 \noindent #1\hfill}\hss}\vss}}\fi%
}

\newcommand{\A}{\;\forall}
\newcommand{\E}{\;\exists}

\newcommand{\ul}{\underline}

\title[]
{Quasianalytic ultradifferentiability cannot be tested in lower dimensions}

\author[A.~Rainer]{Armin Rainer}

\address{University of Education Lower Austria,
Campus Baden M\"uhlgasse 67, A-2500 Baden \&
Fakult\"at f\"ur Mathematik, Universit\"at Wien, 
Oskar-Morgenstern-Platz~1, A-1090 Wien, Austria}
\email{armin.rainer@univie.ac.at}

\begin{document}

\begin{abstract}
	We show that, in contrast to the real analytic case, 
	quasianalytic ultradifferentiability can never be tested in lower dimensions.  
	Our results are based on a construction due to Jaffe.
\end{abstract}

\thanks{The author was supported by the Austrian Science Fund (FWF) Project P 26735-N25.}
\date{\today}

\maketitle

\section{Introduction}

In a recent paper \cite{Bochnak:aa} Bochnak and Kucharz proved that a function on a compact real analytic manifold is real analytic if and only if 
its restriction to every closed real analytic submanifold of dimension two is real analytic. A local version of this theorem 
can be found in \cite{Bochnak:2018aa}.   
It is natural to ask if a similar statements holds in quasianalytic classes of smooth functions $\cC$ which are strictly bigger than the real 
analytic class, but share the property of analytic continuation:
\begin{quote}\it
  	Is a function defined on a $\cC$-manifold of class $\cC$ provided that all its restrictions to $\cC$-submanifolds of lower dimension are 
  	of class $\cC$?
  \end{quote}  
We will show in this paper that the answer to this question is negative for 
all standard quasianalytic \emph{ultradifferentiable} classes defined by growth estimates for the iterated derivatives, 
even if we already know that the function is smooth. 
We shall always assume that the classes $\cC$ are stable under composition and admit an inverse function theorem, consequently,
manifolds of class $\cC$ are well-defined. 

This article is partly motivated by the development of the \emph{convenient setting} for ultradifferentiable function classes in 
\cite{KMRc,KMRq,KMRu} which provides an (ultra)differential calculus for mappings between infinite dimensional locally convex spaces with a 
mild completeness property. 
Typically, the convenient calculus is based on Osgood--Hartogs type theorems which describe objects by ``restrictions'' to 
certain better understood test objects (cf.\ \cite{SpallekTworzewskiWiniarski90}).
While many non-quasianalytic classes can be tested along non-quasianalytic \emph{curves} in the same class \cite{KMRc}, 
the analogous statement is false for quasianalytic classes even if the function in question is smooth. This was shown by Jaffe \cite{Jaffe16} 
for quasianalytic Denjoy--Carleman classes of Roumieu type. In \cite{KMRu} we overcame this problem by testing along all
\emph{Banach plots} in the class (i.e. mappings defined in arbitrary Banach spaces) which raised the question if there is a 
subclass of plots sufficient for recognizing the class. 

The results of this paper show that in finite dimensions quasianalytic $\cC$-plots with lower dimensional domain are never enough 
for testing $\cC$-regularity (even if smoothness is already known). 
In particular, restrictions to $\cC$-submanifolds of lower dimensions cannot recognize $\cC$-regularity.
Actually, we will prove more:
For any $n \ge 2$, any regular quasianalytic class $\cC$, and any positive sequence $N=(N_k)$ 
there exists a function $f \in \cC^\infty(\R^n)$ such that $f \o p \in \cC$ for all $\cC$-plots $p : \R^m \supseteq U \to \R^n$ with 
$m < n$, but 
\[
	\sup_{\substack{x \in K\\\al\in \N^n}} \frac{|\p^\al f(x)|}{\rh^{|\al|} |\al|! N_{|\al|}} = \infty
\] 
for all neighborhoods $K$ of $0$ in $\R^n$ and all $\rh>0$. It will be specified in the next two subsections what we mean here by a 
regular quasianalytic class.

All our results follow from slight modifications of Jaffe's construction.

\subsection{Denjoy--Carleman classes}

Let $U \subseteq \R^n$ be open. 
Let $M=(M_k)$ be a positive sequence.
For $\rh>0$ and $K\subseteq U$ compact consider the seminorm
\[
	\|f\|^M_{K,\rh} := \sup_{\substack{x \in K\\\al\in \N^n}} \frac{|\p^\al f(x)|}{\rh^{|\al|} |\al|! M_{|\al|}}, \quad f \in \cC^\infty(U).  
\]
The \emph{Denjoy--Carleman class of Roumieu type} is defined by 
\[
	\cE^{\{M\}}(U) := \{f \in \cC^\infty(U) : \A \text{ compact } K \subseteq U \E \rh >0 : \|f\|^M_{K,\rh} <\infty\},
\]
and the \emph{Denjoy--Carleman class of Beurling type} by 
\[
	\cE^{(M)}(U) := \{f \in \cC^\infty(U) : \A \text{ compact } K \subseteq U \A \rh >0 : \|f\|^M_{K,\rh} <\infty\},
\]
We shall assume that $M=(M_k)$ is
\begin{enumerate}
	\item logarithmically convex, i.e.\ $M_k^2 \le M_{k-1} M_{k+1}$ for all $k$, and satisfies
	\item $M_0 = 1 \le M_1$ and
	\item $M_k^{1/k} \to \infty$.
\end{enumerate} 
A positive sequence $M=(M_k)$ having these properties (1)--(3) is called a \emph{regular weight sequence}.
The Denjoy--Carleman classes $\cE^{\{M\}}$ and $\cE^{(M)}$ associated with a regular weight sequence $M$ 
are stable under composition and admit a version of the inverse function theorem (cf.\ \cite{RainerSchindl14}).

Let $M=(M_k)$ and $N=(N_k)$ be positive sequences. Then boundedness of the sequence $(M_k/N_k)^{1/k}$ is a sufficient condition for the 
inclusions 
$\cE^{\{M\}} \subseteq \cE^{\{N\}}$ and $\cE^{(M)} \subseteq \cE^{(N)}$ (this means that the inclusions hold on all open sets). 
The condition is also necessary provided that $k! M_k$ is logarithmically convex, see \cite{Thilliez08} and \cite{Bruna80/81}, 
(so in particular if $M$ is a regular weight sequence). 
For instance, stability of the classes $\cE^{\{M\}}$ and $\cE^{(M)}$ by derivation is equivalent to boundedness of $(M_{k+1}/M_k)^{1/k}$ 
(for the necessity we assume that $k! M_k$ is logarithmically convex). 
If $(M_k/N_k)^{1/k} \to 0$ then $\cE^{\{M\}} \subseteq \cE^{(N)}$, and conversely provided that $k! M_k$ is logarithmically convex. 
Hence regular weight sequences $M$ and $N$ are called \emph{equivalent} if 
there is a constant $C>0$ such that $C^{-1} \le (M_k/N_k)^{1/k}\le C$.

For the constant sequence $\mathbf{1} = (1,1,1,\ldots)$ we get the class of real analytic functions $\cE^{\{\bf 1\}} = \cC^\om$ in the Roumieu 
case and the restrictions of entire functions $\cE^{(\bf 1)}$ in the Beurling case.
Note that the conditions (1) and (2) imply that the sequence $M_k^{1/k}$ is increasing. 
Thus, if $M$ satisfies (1) and (2) then the strict inclusions $\cC^\om \subsetneq \cE^{\{M\}}$ and  
$\cC^\om \subsetneq \cE^{(M)}$ are both equivalent to (3)
(for the latter observe that (3) and $\cC^\om = \cE^{(M)}$ would imply that all classes 
$\cC^\om \subseteq \cE^{(\sqrt M)} \subseteq  \cE^{\{\sqrt M\}}  \subseteq \cE^{(M)}$ actually coincide, a contradiction).

A regular weight sequence $M=(M_k)$ is called \emph{quasianalytic} if 
\begin{equation} \label{eq:qaM}
	\sum_k \frac{M_k}{(k+1) M_{k+1}} = \infty.
\end{equation}
By the Denjoy--Carleman theorem, this is the case if and only if the class $\cE^{\{M\}}$ is quasianalytic, 
or equivalently $\cE^{(M)}$ is quasianalytic. See e.g.\ \cite[Theorem 1.3.8]{Hoermander83I} and \cite[Theorem 4.2]{Komatsu73}.

A class $\cC$ of $\cC^\infty$-functions is called \emph{quasianalytic} if 
the restriction to $\cC(U)$ of the map $\cC^\infty(U) \ni f \mapsto T_a f$ which takes $f$ to its infinite Taylor series at $a$ 
is injective for any connected open $U \ni a$. 
For example, the real analytic class $\cC^\om$ has this property and indeed \eqref{eq:qaM} reduces to $\sum_k \frac1{k+1} = \infty$ in this 
case. Further examples of quasianalytic classes $\cE^{\{M\}}$ and $\cE^{(M)}$ that strictly contain 
$\cC^\om$ are given by $M_k := (\log(k+e))^{\de k}$ 
for any $0 < \de \le 1$.

Let $V \subseteq \R^m$ be open.
A mapping $p : V \to U$ of class $\cE^{\{M\}}$ (which means that the component functions $p_j$ are of class $\cE^{\{M\}}$) 
is called a \emph{$\cE^{\{M\}}$-plot in $U$ of dimension $m$}. 
If $m<n$ we say that $p$ is \emph{lower dimensional}.

Now we are ready to state our first results.

\begin{theorem} \label{thm:main}
	Let $M=(M_k)$ be a quasianalytic regular weight sequence. 
	For any $n \ge 2$ and any positive sequence $N=(N_k)$
	there exists a $\cC^\infty$-function $f$ on $\R^n$ of class $\cE^{\{M\}}$ on $\R^n \setminus \{0\}$   
	which does not belong to $\cE^{\{N\}}(\R^n)$, but
	$f \o p \in \cE^{\{M\}}$ for all lower dimensional $\cE^{\{M\}}$-plots $p$ in $\R^n$. 
\end{theorem}

The following Beurling version is an easy consequence; $\cE^{(M)}$-plots are defined in analogy to $\cE^{\{M\}}$-plots.

\begin{theorem} \label{thm:Beurling}
	Let $M=(M_k)$ be a quasianalytic regular weight sequence. 
	For any $n \ge 2$ and any positive sequence $N=(N_k)$
	there exists a $\cC^\infty$-function $f$ on $\R^n$
	of class $\cE^{(M)}$ on $\R^n \setminus \{0\}$ which does not belong to $\cE^{\{N\}}(\R^n)$, but
	$f \o p \in \cE^{(M)}$ for all lower dimensional $\cE^{(M)}$-plots $p$ in $\R^n$.
\end{theorem}

The proofs can be found in \Cref{sec:proofs}.

\begin{remark*}
	The theorems also show that \emph{non-quasianalytic} ultradifferentiability cannot be tested on lower dimensional quasianalytic 
	plots:
	Suppose that $L$ is a non-quasianalytic regular weight sequence, $M \le L$ is a quasianalytic regular weight sequence, 
	and $N$ is an arbitrary positive sequence. 
	By \Cref{thm:main} there is a $\cC^\infty$-function $f$ on $\R^n$ of class $\cE^{\{M\}}$ off $0$ not in $\cE^{\{N\}}(\R^n)$, but 
	of class $\cE^{\{M\}} \subseteq \cE^{\{L\}}$ along every $\cE^{\{M\}}$-plot.
\end{remark*}

\subsection{Braun--Meise--Taylor classes}

Another way to define ultradifferentiable classes which goes back to Beurling \cite{Beurling61} and Bj\"orck \cite{Bjoerck66} and was 
generalized by Braun, Meise, and Taylor \cite{BMT90} is to use weight functions instead of weight sequences. 
By a \emph{weight function} we mean a continuous increasing function $\om : [0,\infty) \to [0,\infty)$ with $\om(0) =0$ 
and 
$\lim_{t \to \infty} \om(t) = \infty$ that satisfies
\begin{enumerate}
   \item $\om(2t) = O(\om(t))$  as  $t \to \infty$, \label{om1}
   \item $\om(t) = O(t)$ as $t \to \infty$, \label{om2}
   \item $\log t = o(\om(t))$ as  $t \to \infty$, and \label{om3}
   \item $\vh(t) := \om(e^t)$  is convex.  \label{om4}	
\end{enumerate}
Consider the \emph{Young conjugate} $\vh^*(t) := \sup_{s\ge 0} \big(st-\vh(s)\big)$, for $t>0$, of $\vh$. 
For compact $K \subseteq U$ and $\rh >0$ consider the seminorm
\[
  \|f\|^\om_{K,\rh} := \sup_{x \in K,\,\al \in \N^n} |\p^\al f(x)| \exp(-\tfrac{1}{\rh} \vh^*(\rh |\al|)), 
  \quad  f \in \cC^\infty(U),
\]
and the ultradifferentiable classes of \emph{Roumieu type} 
\[
	\cE^{\{\om\}}(U) := \{f \in \cC^\infty(U) : \A \text{ compact } K \subseteq U \E \rh >0 : \|f\|^\om_{K,\rh} <\infty\},
\]
and of \emph{Beurling type}  
\[
	\cE^{(\om)}(U) := \{f \in \cC^\infty(U) : \A \text{ compact } K \subseteq U \A \rh >0 : \|f\|^\om_{K,\rh} <\infty\}.
\]
The classes $\cE^{\{\om\}}$ and $\cE^{(\om)}$ are in general not representable by any Denjoy--Carleman class, 
but they are representable (algebraically and topologically) by unions and intersections of 
Denjoy--Carleman classes defined by $1$-parameter families of positive sequences associated with $\om$ \cite{RainerSchindl12}.  
The classes $\cE^{\{\om\}}$ and $\cE^{(\om)}$ are quasianalytic if and only if 
\begin{equation*}
   \int_0^\infty \frac{\om(t)}{1+t^2} \, dt =\infty.
 \end{equation*}

If $\si$ is another weight sequence then $\cE^{\{\om\}} \subseteq \cE^{\{\si\}}$ 
and $\cE^{(\om)} \subseteq \cE^{(\si)}$
if and only if $\si(t) = O(\om(t))$ as $t \to \infty$.
The inclusion $\cE^{\{\om\}} \subseteq \cE^{(\si)}$ holds if and only if $\si(t) = o(\om(t))$ as $t \to \infty$.
For details see e.g.\ \cite{RainerSchindl12}. 
Thus $\om$ and $\si$ are called \emph{equivalent} if $\si(t) = O(\om(t))$ and $\om(t) = O(\si(t))$ as $t \to \infty$.

We will assume that the weight function $\om$ satisfies $\om(t) = o(t)$ as $t \to \infty$ 
which is equivalent to the strict inclusion $\cC^\om = \cE^{\{t\}} \subsetneq \cE^{(\om)}$. 
If $\om$ is equivalent to a concave weight function, then the classes $\cE^{\{\om\}}$ and $\cE^{(\om)}$ are  
stable under composition and 
admit a version of the inverse function theorem (and conversely, see \cite[Theorem 11]{Rainer:aa}). 
They are always stable by derivation.

We shall prove in \Cref{sec:proofs}:

\begin{theorem} \label{thm:main2}
	Let $\om$ be a quasianalytic concave weight function such that $\om(t) = o(t)$ as $t \to \infty$.
	For any $n \ge 2$ and any positive sequence $N=(N_k)$
	there exists a $\cC^\infty$-function $f$ on $\R^n$ 
	of class $\cE^{\{\om\}}$ on $\R^n \setminus \{0\}$
	which does not belong to $\cE^{\{N\}}(\R^n)$, but
	$f \o p \in \cE^{\{\om\}}$ for all lower dimensional $\cE^{\{\om\}}$-plots $p$ in $\R^n$. 
\end{theorem}

\begin{theorem} \label{thm:Beurling2}
	Let $\om$ be a quasianalytic concave weight function such that $\om(t) = o(t)$ as $t \to \infty$.
	For any $n \ge 2$ and any positive sequence $N=(N_k)$
	there exists a $\cC^\infty$-function $f$ on $\R^n$ 
	of class $\cE^{(\om)}$ on $\R^n \setminus \{0\}$ which does not belong to $\cE^{\{N\}}(\R^n)$, but
	$f \o p \in \cE^{(\om)}$ for all lower dimensional $\cE^{(\om)}$-plots $p$ in $\R^n$.
\end{theorem}

$\cE^{\{\om\}}$- and $\cE^{(\om)}$-plots are defined in analogy to $\cE^{\{M\}}$-plots.

\subsection{New quasianalytic classes}

	Let us turn the conditions of the theorems into a definition. 

Let $M=(M_k)$ be any quasianalytic regular weight sequence and 
let $\om$ be any quasianalytic concave weight function with $\om(t) = o(t)$ as $t \to \infty$. 
In the following $\star$ stands for either $\{M\}$, $(M)$, $\{\om\}$, or $(\om)$.

Let $\bar{\cA}_1^{\star}(\R^n)$ be the set of 
all $\cC^\infty$-functions $f$ on $\R^n$ such that $f$ is of class $\cE^{\star}$ along all affine lines in $\R^n$. 
Then $\bar \cA_1^{\star}(\R^n)$ is quasianalytic in the sense that $T_a f = 0$ implies $f=0$ for any $a \in \R^n$.
Indeed, if $f$ is infinitely flat at $a$, then so is the restriction of $f$ to any line $\ell$ through $a$. 
Since the class $\cE^\star$ is quasianalytic, $f|_\ell = 0$ for every line $\ell$ through $a$ and thus $f=0$ on $\R^n$. 
On the other hand $\bar \cA_1^{\star}(\R^n)$ contains $\cE^{\star}(\R^n)$ but is not contained in any 
Denjoy--Carleman class whatsoever, by \Cref{thm:main,thm:Beurling,thm:Beurling2,thm:main2}. 

There are many ways to modify the definition: 
Let $U$ be an open subset of an Euclidean space.
If $\cA_m^{\star}(U)$ is the set of 
all $\cC^\infty$-functions $f$ on $U$ such that $f$ is of class $\cE^{\star}$ along all $\cE^{\star}$-plots in $U$ 
of dimension $m$,
then $\cA_m^{\star}(U)$ is quasianalytic and stable under composition.
Thus $\cA_m^{\star}$-mappings between open subsets of Euclidean spaces form a quasianalytic category askew to all 
Denjoy--Carleman classes. We have strict inclusions
\[
	\cE^{\star}(\R^n) = \cA_n^{\star}(\R^n) \subsetneq \cA_{n-1}^{\star}(\R^n) \subsetneq \cdots \subsetneq \cA_{1}^{\star}(\R^n). 
\]
Indeed the first inclusion is strict by the theorem proved in this paper. That the other inclusions are strict follows 
immediately: if $f \in \cA_{n-1}^{\star}(\R^n) \setminus \cA_{n}^{\star}(\R^n)$ then
$\tilde f(x_1,\dots,x_n,x_{n+1}, \ldots,x_{n+k}):= f(x_1,\dots,x_n) \in \cA_{n-1}^{\star}(\R^{n+k}) \setminus \cA_{n}^{\star}(\R^{n+k})$ 
for all $k \ge 1$.

None of the categories $\cA_m^{\star}$ is cartesian closed:
\[
	\cA_m^{\star}(\R^m,\cA_m^{\star}(\R^m)) \ne \cA_m^{\star}(\R^m \times \R^m) \quad (\text{via } f(x)(y) \mapsto f^\wedge(x,y)).
\] 
In fact, 
the left-hand side equals $\cE^{\star}(\R^m,\cE^{\star}(\R^m))$ and is contained in $\cE^{\star}(\R^m \times \R^m)$, 
by \cite[Theorem 5.2]{KMRu} and \cite{Schindl14a}, which in turn is strictly included in the right-hand side.

Each $\cA_m^{\star}$ is closed under reciprocals: if $f \in \cA^\star_m$ and $f(0) \ne 0$ then $1/f \in \cA^\star_m$ on a 
neighborhood of $0$. This follows from stability under composition and the fact that $x \mapsto 1/x$ is real analytic off $0$. 

Suppose that $\cE^\star$ is stable under differentiation. If $f \in \cA^\star_m$ then $d_v^k f \in \cA^\star_{m-1}$ for all 
$m\ge 2$, all vectors $v$, and all $k$, thus also $\p^\al f \in \cA^\star_{m-1}$ for all multi-indices $\al$. Indeed, if 
$p$ is a $\cE^*$-plot of dimension $m-1$, then
\[
	 d_v^k f(p(s) + t v) =  \p_t^k \big( f(p(s) + t v) \big)
\] 
is of class $\cE^\star$ in $s$ for all $t$, since $(s,t) \mapsto p(s) + t v$ is an $\cE^\star$-plot of dimension $m$ and 
$\cE^\star$ is stable under differentiation. 

Another interesting stability property of $\cA^\star_1$ and $\bar \cA^\star_1$, under the assumption 
that $\cE^\star$ is stable under differentiation, is the following: 
Assume that the coefficients of a polynomial 
\[
	\vh(x,y) = y^d + a_1(x) y^{d-1} + \cdots + a_d(x)
\] 
are germs of $\cA^\star_1$ (resp.\ $\bar \cA^\star_1$) functions at $0$ in $\R^n$ and $h$ is germ of a $\cC^\infty$-function at $0$ 
such that $\vh(x,h(x))=0$. Then $h$ is actually also a germ of a $\cA^\star_1$ (resp.\ $\bar \cA^\star_1$) function. 
This follows immediately from the case $n=1$ due to \cite{Thilliez10}; in this reference only the case $\star=\{M\}$ was treated, 
but the arguments apply to all cases. It seems to be unknown whether a similar result holds for $\cE^\star$ and $n>1$, but 
see \cite{Belotto-da-Silva:2017aa}.

\section{Proofs} \label{sec:proofs}

\subsection{Proof of \texorpdfstring{\Cref{thm:main}}{Theorem 1}}
The proof is based on a construction due to Jaffe \cite{Jaffe16}.

\begin{lemma}[{\cite[Proposition 5.2]{Jaffe16}}] \label{lem:Jaffe52}
	Let $M$ be a regular weight sequence.
	For any integer $n\ge 2$ there exists a function $f \in \cE^{\{M\}}(\R^n)$ with the following properties: 
	there is a constant $B=B(n)$ such that for all compact $K \subseteq \R^n$ and all $\al \in \N^n$ 	
	\begin{align*}
		|\p^\al f(x)| &\le B^{|\al|} (|K|+1)^{|\al|} |\al|! M_{|\al|} \quad \text{ for all }  x \in K,
		\\
		|\p^\al f(x)| &\le B^{|\al|} (|K|+1)^{|\al|} |\al|! \big(1+ |x|^{-2(|\al|+1)}\big) \quad \text{ for all }  
		x \in K \setminus \{0\},
	\end{align*}
	and for all $k \ge 1$ and $i=1,\ldots,n$
	\[
		\Big|\frac{\p^{2k} f}{\p x_i^{2k}}(0)\Big| \ge  \frac{(2k)! M_{k}}{2^k}.
	\]
	Here $|K| := \sup_{x \in K} |x|$.	
\end{lemma}

It is not hard to see that the fact that $M$ is logarithmically convex, or equivalently, $m_k := M_{k+1}/M_k$ is increasing, implies that 
\[
	M_k =  \frac{m_k^{k+1}}{\vh(m_k)}, \quad \text{ where  } \vh(t) := \sup_{k\ge 0} \frac{t^{k+1}}{M_k}.
\]
This can be used to see that 
\[
	f(x) := \sum_{k=1}^\infty 2^{-k} \vh(m_k)^{-1} \big( x - i/m_k \big)^{-1}
\]
defines a smooth function on $\R$ with $\|f^{(k)}\|_{L^\infty} \le k! M_k$, $|f^{(k)}(x)| \le k!/|x|^{k+1}$ if $x\ne 0$ 
and $|f^{(k)}(0)| \ge k! M_k/2^k$
 for all $k$. Composing $f$ with the squared Euclidean norm in $\R^n$ gives a function with the properties in the lemma.
 For details see \cite{Jaffe16}.

Let $\vh: [0,1] \to [0,1]$ be a strictly monotone infinitely flat smooth surjective function with $\vh(t) \le t$ for all $t \in [0,1]$. 
Let $\vh_{[n]} := \vh \o \vh_{[n-1]}$, $n\ge 1$, with $\vh_{[0]}:= \on{Id}$ denote the iterates of $\vh$. 
Consider the arc 
\[
	A := \big\{\Ph(t) := (t, \vh(t), \vh_{[2]}(t), \cdots, \vh_{[n-1]}(t)) : t \in (0,1)\big\} \subseteq \R^n.
\]
Note that $t \ge \vh(t) \ge \cdots \ge \vh_{[n-1]}(t)$ for all $t$. 

Without loss of generality we may assume that the sequence $M_k^{1/k}$ is \emph{strictly} increasing \cite[Lemma 4.3]{Jaffe16}.
We define a sequence of points $a_k$ in $A$ by fixing the $n$-th coordinate of $a_k$ to 
\[
	(a_k)_n := M_k^{-1/(4k)}.
\]

For each $\ell \in \N_{\ge 1}$ define a sequence $M^{(\ell)} = (M^{(\ell)}_k)$ by 
\[
	M^{(\ell)}_k := 
	\begin{cases}
		1 & \text{ if } 0 \le k < \ell,
		\\
		c_\ell^{2k - 2\ell +1} M_{k} & \text{ if } k \ge \ell,
	\end{cases}
\]
where $c_\ell \ge M_\ell$ are constants to be determined below. 
Notice that each $M^{(\ell)}$ is a regular weight sequence equivalent to $M$.

By \Cref{lem:Jaffe52}, for each $\ell \in \N_{\ge1}$ there is a function $f_\ell \in \cE^{\{M^{(\ell)}\}}(\R^n) = \cE^{\{M\}}(\R^n)$ 
such that for all compact $K \subseteq \R^n$ and all $\al \in \N^n$ 
we have	(for $a := 1 + \sup_\ell |a_\ell|$ )
	\begin{align}
		|\p^\al f_\ell(x)| &\le B^{|\al|} (|K|+a)^{|\al|} |\al|! M^{(\ell)}_{|\al|} \quad \text{ for all }  x \in K, \label{est1}
		\\
		|\p^\al f_\ell(x)| &\le B^{|\al|} (|K|+a)^{|\al|} |\al|! \big(1+|x-a_\ell|^{-2(|\al|+1)} \big) \quad \text{ for all }  
		x \in K \setminus \{a_\ell\}, \label{est2}
	\end{align}
	where $B=B(n)$,
	and for all $k \ge 1$
	\begin{equation} \label{est3}
		\Big|\frac{\p^{2k} f_\ell}{\p x_1^{2k}}(a_\ell)\Big| \ge   \frac{(2k)! M^{(\ell)}_{k}}{2^k}.
	\end{equation}
Define
\[
	f := \sum_{\ell=1}^\infty 2^{-\ell} f_\ell.
\]
It is easy to check that $f$ is $\cC^\infty$ on $\R^n$ and of class $\cE^{\{M\}}$ on $\R^n\setminus \{0\}$.

Note that $f$ depends on the choice of the coefficients $c_\ell$.
Next we will show that, given any positive sequence $N=(N_k)$, we may choose the constants $c_\ell$ and hence $f$ in 
such a way that $f$ does not belong to $\cE^{\{N\}}$ in any neighborhood of the origin.

\begin{lemma} \label{lem:blowup}
	The constants $c_\ell \ge M_\ell$ can be chosen such that for all $k\ge 1$
	\[
		\Big|\frac{\p^{2k} f}{\p x_1^{2k}}(a_k)\Big| \ge  (2k)!  M_{2k} N_{2k}.
	\]
\end{lemma}

\begin{proof}
Since $M^{(k)}_{k}= c_k M_k$, \eqref{est2} and \eqref{est3} give 
\begin{align*}
	\Big|\frac{\p^{2k} f}{\p x_1^{2k}}(a_k)\Big| \ge 4^{-k}  (2k)! c_k M_{k}  - 
	\sum_{\ell \ne k} 2^{-\ell} B^{2k} (|K|+a)^{2k} (2k)! \big(1+|a_k-a_\ell|^{-2(2k+1)} \big).
\end{align*}
The sum on the right-hand side is bounded by a constant (depending on $k$) since the sequence $M_k^{1/k}$ is strictly increasing and hence 
$\inf_{\ell\ne k} |a_k - a_\ell|>0$. The assertion follows easily.
\end{proof}

\Cref{lem:blowup} implies that $f$ cannot be of class $\cE^{\{N\}}$ in any neighborhood of the origin. 
Otherwise there would be constants $C,\rh>0$ such that, for large $k$,
\[
	(2k)!  M_{2k} N_{2k} \le \Big|\frac{\p^{2k} f}{\p x_1^{2k}}(a_k)\Big| \le C \rh^{2k} (2k)!  N_{2k}  
\]
which leads to a contradiction as $M_k^{1/k} \to \infty$.

It remains to show that $f \o p \in \cE^{\{M\}}(V)$ for any $\cE^{\{M\}}$-plot $p : V \to \R^n$, 
where $V \subseteq \R^m$ with $m < n$. 
We will use the following lemma.

\begin{lemma} \label{lem:estimate}
	Let $K \subseteq \R^n \setminus \{a_k\}_k$ be a compact set such that 
	\[
		\on{dist}(a_k, K) \ge M_k^{-1/(4k)} \quad \text{ for all } k > k_0. 
	\]
	Then there exists $\rh>0$ such that $\|f\|^M_{K,\rh} <\infty$. 
	Neither $\rh$ nor $\|f\|^M_{K,\rh}$ depend on the choice of the constants $c_\ell$.
\end{lemma}

\begin{proof}
	For $x \in K$ and $|\al|\ge 1$,
	\begin{align*}
	|\p^\al f(x)| &\le \sum_{\ell=1}^\infty 2^{-\ell} |\p^\al f_\ell(x)| = 
	\sum_{\ell=1}^{|\al|} 2^{-\ell} |\p^\al f_\ell(x)| + \sum_{\ell=|\al|+1}^\infty 2^{-\ell} |\p^\al f_\ell(x)|.  
\end{align*}
	By \eqref{est1} and the definition of $M^{(\ell)}$, the second sum is bounded by $B^{|\al|} (|K|+a)^{|\al|} |\al|!$. 
	For the first sum we have, by \eqref{est2},
	\begin{align*}
		\sum_{\ell=k_0+1}^{|\al|} 2^{-\ell} |\p^\al f_\ell(x)|
		&\le 
		B^{|\al|} (|K|+a)^{|\al|} |\al|! \sum_{\ell=k_0+1}^{|\al|} 2^{-\ell} \big(1+|x-a_\ell|^{-2(|\al|+1)} \big)
		\\
		&\le 
		B^{|\al|} (|K|+a)^{|\al|} |\al|! M_{|\al|} \sum_{\ell=k_0+1}^{|\al|} 1
		\\
		&\le 
		(eB)^{|\al|} (|K|+a)^{|\al|} |\al|! M_{|\al|}.
	\end{align*}
	A similar estimate holds for $\sum_{\ell=1}^{k_0} 2^{-\ell} |\p^\al f_\ell(x)|$ since 
	$\on{dist}(a_k, K) \ge \ep >0$ for all $k\le k_0$. 
\end{proof}

Let $p=(p_1,\ldots,p_n) : V \to \R^n$ be an $\cE^{\{M\}}$-plot, where $V \subseteq \R^m$ is a neighborhood of the origin and $m < n$.

\begin{lemma} \label{lem:resolution}
	There is a compact neighborhood $L \subseteq V$ of $0$ such that $K:=p(L)$ satisfies 
	\begin{equation} \label{claim}
		\on{dist}(\Ph(t),K) \ge \vh_{[n-1]}(t) \quad \text{ for all small } t>0.	
	\end{equation}	 
\end{lemma}

\begin{proof}
	We may assume that no component $p_j$ vanishes identically; indeed, if $p_j \equiv 0$ then 
	$K$ is contained in the coordinate plane $y_j = 0$ and hence 
	$\on{dist}(\Ph(t),K) \ge \vh_{[j-1]}(t) \ge \vh_{[n-1]}(t)$ for all $t$.

	Suppose that $p(0) \ne 0$. Then there exists a compact neighborhood $L$ of $0$ such that 
	$\on{dist}(0,K) =:\ep >0$, where $K= p(L)$. For sufficiently small $t>0$ we have $|\Ph(t)| \le \ep/2$. 
	For such $t$, 
	\[
	\on{dist}(\Ph(t),K) \ge \on{dist}(0,K) - |\Ph(t)| \ge \ep/2 \ge |\Ph(t)|\ge \vh_{[n-1]}(t).
	\]

	Assume that $p(0) = 0$ and 	
	that $p_j(x) = x^{\al_j} u_j(x)$ for $j=1,\ldots,n$, where $x=(x_1,\ldots,x_m)$, all $u_j$ are non-vanishing 
	and the set of exponents $\{\al_1,\ldots, \al_n\} \subseteq \N^m$ is totally ordered with respect to the 
	natural partial order of multiindices 
	(that is, for all $1 \le i,j \le n$ we have $\al_i \le \al_j$ or $\al_j \le \al_i$).
	Let $\be_1 \le \be_2 \le  \ldots \le \be_n$ be an ordered arrangement of $\{\al_1,\ldots, \al_n\}$.
	Let $m_i$ be the number of zero components of $\be_i$, for $i =1,\ldots,n$.
	Since $p(0)=0$, we have $m_1 \le m-1$. On the other hand $m_i \ge m_{i+1}$ for all $i = 1,\ldots,n-1$. 
	Since $m<n$, we must have $m_{i_0} = m_{i_0+1}$ for some $i_0$.
	That means there exist two distinct numbers $i,j \in \{1,\ldots,n\}$ with $\al_i \le \al_j$ 
	such that $\al_i$ and $\al_j$ have the same number of zero components. 
	Thus we may find a positive integer $d$ such that $d \cdot \al_i \ge \al_j$.
	Consequently, there is a constant $C>0$ such that for all $x$ in a neighborhood $L$ of $0 \in \R^m$,
		\[
			|p_j(x)| \le C\, |p_i(x)| \quad \text{ and }\quad |p_i(x)|^d \le C\, |p_j(x)|.  
		\]
	This implies that $K = p(L)$ satisfies \eqref{claim}. 
	In fact, the $i$-th component of $\Ph(t)$ is $\vh_{[i-1]}(t)$ and the $j$-th component 
	is $\vh_{[j-1]}(t) = \vh_{[j-i]}\big( \vh_{[i-1]}(t) \big)$. 
	Since $\vh_{[j-i]}$ is an infinitely flat function while $K$ is contained in the set $\{C^{-1} |y_i|^d \le |y_j| \le C |y_i| \}$,  
	$\on{dist}(\Ph(t),K)$ is larger than $\vh_{[j-1]}(t)$ for all sufficiently small $t>0$.

	The general situation can be reduced to these special cases by the desingularization theorem 
	\cite[Theorem 5.12]{BM04} using \cite[Lemma 7.7]{BM04} in order to get the exponents totally ordered. 
	Indeed, applying \cite[Theorem 5.12]{BM04} to the product of all nonzero $p_j$ and all nonzero
	differences of any two $p_i,p_j$ we may assume that after pullback by a suitable mapping $\si$ 
	the components $p_j$ are locally a monomial times a nonvanishing factor (in suitable coordinates),
	and the collection of exponents of the monomials is totally ordered.  
	Here we apply the desingularization theorem to the quasianalytic class 
	$\cC = \bigcup_{k\in \N} \cE^{\{M^{+k}\}}$, where $M^{+k}$ is the regular weight sequence defined 
	by $M^{+k}_j := M_{j+k}$, which has all required properties. 
	This is necessary since the class $\cE^{\{M\}}$ might not be closed under differentiation.
\end{proof}

\begin{remark} \label{rem:resolution}
	For later reference we note that \Cref{lem:resolution} holds for all lower dimensional $\cC$-plots, 
	where $\cC$ is any quasianalytic class of smooth functions which contains the restrictions of polynomials, 
	is stable by composition, differentiation, division by coordinates, and admits an inverse function theorem; 
	cf.\ \cite{BM04}.  
\end{remark}

	Now we can prove that $f \o p \in \cE^{\{M\}}(V)$ for any lower dimensional $\cE^{\{M\}}$-plot $p : V \to \R^n$. 
	To be of class $\cE^{\{M\}}$ is a local condition. So we may assume without loss of generality that 
	$V$ is a neighborhood of $0$.
	By \Cref{lem:resolution}, we may further assume that (after shrinking) $V= L$ is a compact neighborhood of $0$ 
	such that $K=p(L)$ satisfies \eqref{claim}.  
	By \Cref{lem:estimate}, there exists $\rh >0$ such that $\|f\|^M_{K,\rh} =: C < \infty$.
	Since $p \in \cE^{\{M\}}$, there exists $\si>0$ such that $\|p\|^M_{L,\si} =: D < \infty$. 
	Logarithmic convexity of $M$ implies $M_1^k M_k \ge M_j M_{\al_1} \cdots M_{\al_j}$ for all 
	$\al_i \in \N_{>0}$ with $\al_1 + \cdots + \al_j = k$ 
	(cf.\ \cite[Lemma 2.9]{KMRc}).
	Consequently, in view of the Fa\'a di Bruno formula, for $k>0$ and $x \in L$,
	\begin{align*}
		\frac{\|(f \o p)^{(k)}(x)\|_{L^k(\R^m,\R)}}{k!} &\le \sum_{j\ge 1} \sum_{\al_i} \frac{\|f^{(j)}(p(x))\|_{L^j(\R^n,\R)}}{j!} 
		\prod_{i=1}^j \frac{\|p^{(\al_i)}(x)\|_{L^{\al_i}(\R^m,\R^n)}}{\al_i!}
		\\
		&\le \sum_{j\ge 1} \sum_{\al_i} C \rh^j M_j 
		\prod_{i=1}^j D \si^{\al_i} M_{\al_i}
		\\
		&\le  C (M_1\si)^k M_k \sum_{j\ge 1} \binom{k-1}{j-1}  (D\rh)^j
		\\
		&\le  CD\rh (M_1\si)^k   (1+ D\rh)^{k-1} M_k,
	\end{align*}
	that is, there exists $\ta>0$ such that $\|f \o p \|^M_{L,\ta} < \infty$.
	This ends the proof of \Cref{thm:main}.

\subsection{Proof of \texorpdfstring{\Cref{thm:Beurling}}{Theorem 2}}

	Set $L_k := M_k^{1/2}$. Then $L=(L_k)$ is a quasianalytic regular weight sequence satisfying $(L_k/M_k)^{1/k} \to 0$.
	\Cref{thm:main} associates a function $f$ with $L$ which is as required. 
	Indeed, $f$ is of class $\cE^{\{L\}}\subseteq \cE^{(M)}$ along the image of lower dimensional $\cE^{(M)}$-plots $p$, 
	by \Cref{lem:resolution} and \Cref{rem:resolution},
	and thus $f\o p$ is $\cE^{(M)}$, since the class is stable by composition.

\subsection{Proof of \texorpdfstring{\Cref{thm:main2}}{Theorem 3}}

By \cite[Theorem 11]{Rainer:aa}, there is a family $\fM$ of quasianalytic regular weight sequences $M=(M_k)$  
such that 
\[
	\cE^{\{\om\}}(U) = \{ f \in \cC^\infty(U) : \A \text{ compact } K \subseteq U \E M \in \fM \E \rh >0 : \|f\|^M_{K,\rh} < \infty\}.
\]
Fix $M \in \fM$ and a positive sequence $N=(N_k)$.  
Let $f$ be the $\cC^\infty$-function associated with $M$ and $N$ provided by \Cref{thm:main}. 
Then $f$ is not of class $\cE^{\{N\}}$.
Let $p$ be any lower dimensional $\cE^{\{\om\}}$-plot. 
Then $f \o p$ is of class $\cE^{\{\om\}}$, by \Cref{lem:resolution} and \Cref{rem:resolution}, since 
$\cE^{\{\om\}}$ is stable under composition as $\om$ is concave.

\subsection{Proof of \texorpdfstring{\Cref{thm:Beurling2}}{Theorem 4}}

By \cite{RainerSchindl12},
there is a one-parameter family $\fM = \{M^x\}_{x>0}$ family of quasianalytic positive sequences with $(M^x_k)^{1/k} \to \infty$ for all $x$,  
$M^x \le M^y$ if $x \le y$, and 
\[
	\cE^{(\om)}(U) = \cE^{(\fM)}(U) :=  \bigcap_{x>0} \cE^{(M^x)}(U).  
\]
The next lemma is inspired by \cite[Lemma 6]{Komatsu79b}.

\begin{lemma} \label{lem:Beurling2}
	There is a quasianalytic regular weight sequence $L$ such that $(L_k/M^{x}_k)^{1/k} \to 0$ for all $x>0$.
\end{lemma}

\begin{proof}
	Choose a positive sequence $x_p$ which is strictly decreasing to $0$. 
	For every $p \ge 1$ we know that $(M^{x_p}_k)^{1/k} \to \infty$ as $k \to \infty$. Thus for every $p$ there is a constant $C_p>0$ 
	such that 
	\[
		\frac{1}{(M^{x_p}_k)^{1/k}} \le \frac{C_p^{1/k}}{p} \quad \text{ for all } k.  
	\]
	Choose a strictly increasing sequence $j_p$ of positive integers such that $C_p \le 2^{j_p}$ for all $p$. 
	Consider the sequence $L$ defined by $L_j := 1$ if $j<j_1$ and  
	\[
		L_j := \sqrt{M^{x_p}_j} \quad \text{ if } j_p \le j < j_{p+1}.  
	\]
	First, for $j_p \le j < j_{p+1}$,
	\begin{align*}
		L_j^{1/j} = \sqrt{(M^{x_p}_j)^{1/j}} \ge \sqrt{\frac{p}{C_p^{1/j}}} \ge \sqrt{\frac{p}{2}} 
	\end{align*}
	which tends to infinity as $j \to \infty$.
	On the other hand, for $j_p \le j < j_{p+1}$ and $x_p \le x$, 
	\begin{align*}
		\Big(\frac{L_j}{M^x_j}\Big)^{1/j} = \Big(\frac{\sqrt{M^{x_p}_j}}{M^x_j}\Big)^{1/j} \le \frac{1}{\sqrt{(M^x_j)^{1/j}}}  
	\end{align*}
	which tends to $0$ as $j \to \infty$. 

	Let $\ul L$ be the log-convex minorant of $L$. Since $L_k^{1/k} \to \infty$, there exists a sequence $k_j \to \infty$ of integers 
	such that $\ul L_{k_j} = L_{k_j}$ for all $j$. It follows that $\ul L_k^{1/k} \to \infty$, since $\ul L_k^{1/k}$ is increasing 
	by logarithmic convexity.
	Thus $\ul L$ has all required properties.
\end{proof}

The proof of \Cref{thm:Beurling2} now follows the arguments in the proof of \Cref{thm:Beurling}. 
\Cref{thm:main} associates a function $f$ with the sequence $L$ from 
\Cref{lem:Beurling2}
 which is of class $\cE^{\{L\}}$ along the image of any lower dimensional $\cE^{(\om)}$-plots $p$ (by \Cref{lem:resolution} 
 and \Cref{rem:resolution}). 
Since $(L_k/M^x_k)^{1/k} \to 0$ for all $x>0$, we have an inclusion of classes $\cE^{\{L\}} \subseteq \cE^{(\om)}$. 
Since $\om$ is concave, the class $\cE^{(\om)}$ is stable under composition, whence 	
$f\o p$ is of class $\cE^{(\om)}$.
The proof of \Cref{thm:Beurling2} is complete.


\begin{thebibliography}{10}

\bibitem{Belotto-da-Silva:2017aa}
A.~Belotto~da Silva, I.~Biborski, and E.~Bierstone, \emph{Solutions of
  quasianalytic equations}, Selecta Math. (N.S.) \textbf{23} (2017), no.~4,
  2523--2552. 

\bibitem{Beurling61}
A.~Beurling, \emph{Quasi-analyticity and general distributions}, Lecture notes,
  AMS Summer Institute, Stanford, 1961.

\bibitem{BM04}
E.~Bierstone and P.~D. Milman, \emph{Resolution of singularities in
  {D}enjoy-{C}arleman classes}, Selecta Math. (N.S.) \textbf{10} (2004), no.~1,
  1--28.

\bibitem{Bjoerck66}
G.~Bj{\"o}rck, \emph{Linear partial differential operators and generalized
  distributions}, Ark. Mat. \textbf{6} (1966), 351--407.

\bibitem{Bochnak:aa}
J.~Bochnak and W.~Kucharz, \emph{Real analyticity is concentrated in dimension
  2}, Oberwolfach Preprint OWP-2018-23, DOI:10.14760/OWP-2018-23.

\bibitem{Bochnak:2018aa}
J.~Bochnak and J.~Siciak, \emph{A characterization of analytic functions of
  several variables}, Ann. Polon. Math. (2018), DOI:10.4064/ap180119-26-3.

\bibitem{BMT90}
R.~W. Braun, R.~Meise, and B.~A. Taylor, \emph{Ultradifferentiable functions
  and {F}ourier analysis}, Results Math. \textbf{17} (1990), no.~3-4, 206--237.

\bibitem{Bruna80/81}
J.~Bruna, \emph{On inverse-closed algebras of infinitely differentiable
  functions}, Studia Math. \textbf{69} (1980/81), no.~1, 59--68.

\bibitem{Hoermander83I}
L.~H{\"o}rmander, \emph{The analysis of linear partial differential operators.
  {I}}, Grundlehren der Mathematischen Wissenschaften [Fundamental Principles
  of Mathematical Sciences], vol. 256, Springer-Verlag, Berlin, 1983,
  Distribution theory and Fourier analysis.

\bibitem{Jaffe16}
E.~Y. Jaffe, \emph{Pathological phenomena in {D}enjoy-{C}arleman classes},
  Canad. J. Math. \textbf{68} (2016), no.~1, 88--108. 

\bibitem{Komatsu73}
H.~Komatsu, \emph{Ultradistributions. {I}. {S}tructure theorems and a
  characterization}, J. Fac. Sci. Univ. Tokyo Sect. IA Math. \textbf{20}
  (1973), 25--105.

\bibitem{Komatsu79b}
\bysame, \emph{An analogue of the {C}auchy-{K}owalevsky theorem for
  ultradifferentiable functions and a division theorem for ultradistributions
  as its dual}, J. Fac. Sci. Univ. Tokyo Sect. IA Math. \textbf{26} (1979),
  no.~2, 239--254.

\bibitem{KMRc}
A.~Kriegl, P.~W. Michor, and A.~Rainer, \emph{The convenient setting for
  non-quasianalytic {D}enjoy--{C}arleman differentiable mappings}, J. Funct.
  Anal. \textbf{256} (2009), 3510--3544.

\bibitem{KMRq}
\bysame, \emph{The convenient setting for quasianalytic {D}enjoy--{C}arleman
  differentiable mappings}, J. Funct. Anal. \textbf{261} (2011), 1799--1834.

\bibitem{KMRu}
\bysame, \emph{The convenient setting for {D}enjoy--{C}arleman differentiable
  mappings of {B}eurling and {R}oumieu type}, Rev. Mat. Complut. \textbf{28}
  (2015), no.~3, 549--597. 

\bibitem{Rainer:aa}
A.~Rainer and G.~Schindl, \emph{On the extension of {W}hitney ultrajets, {II}},
  ar{X}iv:1808.10253.

\bibitem{RainerSchindl12}
\bysame, \emph{Composition in ultradifferentiable classes}, Studia Math.
  \textbf{224} (2014), no.~2, 97--131.

\bibitem{RainerSchindl14}
\bysame, \emph{Equivalence of stability properties for ultradifferentiable
  function classes}, Rev. R. Acad. Cienc. Exactas Fis. Nat. Ser. A Math.
  RACSAM. \textbf{110} (2016), no.~1, 17--32.

\bibitem{Schindl14a}
G.~Schindl, \emph{The convenient setting for ultradifferentiable mappings of
  {B}eurling- and {R}oumieu-type defined by a weight matrix}, Bull. Belg. Math.
  Soc. Simon Stevin \textbf{22} (2015), no.~3, 471--510. 

\bibitem{SpallekTworzewskiWiniarski90}
K.~Spallek, P.~Tworzewski, and T.~Winiarski, \emph{Osgood-{H}artogs-theorems of
  mixed type}, Math. Ann. \textbf{288} (1990), no.~1, 75--88. 

\bibitem{Thilliez08}
V.~Thilliez, \emph{On quasianalytic local rings}, Expo. Math. \textbf{26}
  (2008), no.~1, 1--23.

\bibitem{Thilliez10}
\bysame, \emph{Smooth solutions of quasianalytic or ultraholomorphic
  equations}, Monatsh. Math. \textbf{160} (2010), no.~4, 443--453.

\end{thebibliography}

\def\cprime{$'$}
\providecommand{\bysame}{\leavevmode\hbox to3em{\hrulefill}\thinspace}
\providecommand{\MR}{\relax\ifhmode\unskip\space\fi MR }
\providecommand{\MRhref}[2]{%
  \href{http://www.ams.org/mathscinet-getitem?mr=#1}{#2}
}
\providecommand{\href}[2]{#2}

\end{document}